\documentclass[reqno]{amsart}

\usepackage[bookmarks,colorlinks=true,citecolor=black,linkcolor=black,pagecolor=black,urlcolor=black,pagebackref=true]{hyperref}

\usepackage{amsmath,amssymb,amsthm}
\usepackage{mathabx}
\usepackage{wasysym}
\usepackage{mathtools}

\usepackage[size=small,width=\linewidth]{caption}
\usepackage{fmtcount}

\usepackage{epsfig,psfrag,color}
\definecolor{darkgreen}{rgb}{0.00,0.50,0.10}
\definecolor{lightgreen}{rgb}{0.20,0.70,0.30}

\usepackage{enumitem}
\usepackage{booktabs}
\usepackage{stmaryrd}

\newtheorem{theorem}{Theorem}
\newtheorem{lemma}[theorem]{Lemma} 
 
\newtheorem{proposition}[theorem]{Proposition}

\newtheorem{definition}[theorem]{Definition}

\newtheorem{conjecture}[theorem]{Conjecture}

\newcommand{\By}[2]{\overset{\mbox{\tiny{#1}}}{#2}}

\newcommand{\F}{\mathbb{F}}
\newcommand{\N}{\mathbb{N}}

\newcommand{\R}{\mathbb{R}}

\newcommand{\upe}{\mathrm{e}}

\newcommand{\upC}{\mathrm{C}}

\newcommand{\upK}{\mathrm{K}}

\newcommand{\upV}{\mathrm{V}}

\newcommand{\upZ}{\mathrm{Z}}

\newcommand{\Prob}{\textup{\textsf{P}}}

{\begin{list}{}%
         {\setlength{\leftmargin}{#1}}%
         \item[]%
}
{\end{list}}

\title[Hamilton space of random graphs]{When Hamilton circuits generate 
the cycle space of a random graph}

\author{Peter Heinig}

\address{Zentrum~Mathematik, M9, Technische~Universit{\"a}t~M{\"u}nchen, \newline
Boltzmannstra{\ss}e~3, D-85747 Garching~bei~M{\"u}nchen, Germany} 

\email{heinig@ma.tum.de}

\begin{document}

\begin{abstract}
If $\varepsilon>0$ and $p(n)\geq n^{-\frac12+\varepsilon}$, 
in a binomial random graph $G_{n,p}$ a.a.s. the set of cycles which can be 
constructed as a symmetric difference of Hamilton circuits is as large as 
parity by itself permits (all cycles if $n$ is odd, all even cycles if $n$ is even). 
Moreover, every $p$ which ensures the above property a.a.s. must necessarily be 
such that for any constant $c>0$, 
eventually $p(n)\geq\tfrac{\log n + 2\log\log n + c}{n}$. 
So whatever the smallest sufficient $p$ for an a.a.s. 
Hamilton-generated cycle space might be, 
it does not coincide with the threshold for hamiltonicity of $G_{n,p}$. 

Hamilton circuit, Hamilton-connected, 
consequences of structured spanning subgraphs

\end{abstract}

\maketitle 

\vspace{-2em}

\section{Introduction}

In the face of the apparent intractability of efficiently describing the 
set of all finite hamiltonian graphs, one realistic strategy to acquire more 
knowledge is the study of \emph{slices} of that complicated set: under the 
assumption that a known sufficient condition for hamiltonicity holds, one tries to 
find proofs that such a graph has \emph{extra properties}, 
such as having \emph{many} Hamilton circuits, 
or having many \emph{well-distributed} Hamilton circuits; 
examples in the deterministic setting (with the sufficient 
condition being a mininum-degree-condition) are to be found in \cite{MR2520274} and 
\cite{KUEHNLAPINSKASOSTHUSoptimalpackingsofhamiltoncyclesingraphsofhighminimumdegree}. 
In \cite[Theorem~1]{arXiv:1112.5101v1}, another extra property, 
a sort of `richness-' or `universality-' property (the set of all Hamilton circuits 
\emph{generates the cycle space}), previously known for certain Cayley-graphs 
only \cite{MR1057481}, was proved to hold in graphs of minimum degree slightly 
above Dirac's bound for hamiltonicity.

There are known examples for the above paradigm in a random setting as
well (e.g., \cite{MR860585} \cite{MR1164841} \cite{MR1094121} \cite{MR2877561}), 
and in this note we give another variation on that theme: we will combine some 
recent results in order to prove a new result, 
Theorem~\ref{540453.1053.0531.05432.06531.05120}, teaching us that sufficiently 
dense binomial random graphs a.a.s. have the above universality-property 
(and for a vastly smaller edge-probability than what would a.a.s. force the 
minimum-degree-condition from \cite[Theorem~1]{arXiv:1112.5101v1} to hold). 

We write $\upZ_1(G;\F_2)$ for the cycle space of a graph $G$ in the 
standard graph-theoretical sense, and $\mathcal{H}(G)$ for the set of 
all Hamilton circuits in $G$. The word `circuit' means `$2$-regular connected graph', 
while `cycle' means `element of $\upZ_1(G;\F_2)$'. 
For any graph $G$, the notation $\upZ_1(G;\F_2) = \langle\mathcal{H}(G)\rangle_{\F_2}$, 
in which $\langle \mathcal{H}(G)\rangle_{\F_2}$ denotes the $\F_2$-span of 
the (vectors w.r.t the standard basis of the edge-space describing the) 
Hamilton circuits of $G$, is equivalent to saying `every cycle of $G$ 
is a symmetric difference of Hamilton circuits of $G$'. 

We will phrase Theorem~\ref{540453.1053.0531.05432.06531.05120} in linear-algebraic 
language, to be able to efficiently use auxiliary results 
from \cite{arXiv:1112.5101v1}. However, the results can easily be translated 
into more constructive phrasings: 
\ref{653.0654312.063.0654312.0} and \ref{45613.0465130.4531.04563120} are 
equivalent to the statement that, \emph{a.a.s.}, the set of cycles which 
can be constructed as a symmetric difference of Hamilton circuits is as 
large as parity alone permits. I.e., \emph{all} cycles if $n$ is odd, 
and all \emph{even} cycles (i.e., elements of the cycle space having a 
support of even size, of which circuits are examples) if $n$ is even. 
One of these equivalences is a deterministic one: 
statement \ref{45613.0465130.4531.04563120} is deterministically 
equivalent to saying that for odd $n$, every cycle of $G$ 
is a symmetric difference of (edge-sets of) Hamilton-circuits of $G$. 
The other implication holds in an a.a.s. sense. 
Statement \ref{653.0654312.063.0654312.0} is not deterministically 
equivalent to saying that for even $n$, every even cycle 
is a symmetric difference of Hamilton circuits: 
if $n$ is even and $G$ happens to be \emph{bipartite}, then every 
even cycle being a symmetric difference of Hamilton circuits is equivalent to 
$\dim_{\F_2}\left(\upZ_1(G;\F_2)/\langle\mathcal{H}(G)\rangle_{\F_2}\right) = 0$, 
not $=1$. But when restricted to \emph{non-bipartite} graphs only---and $G_{n,p}$ with 
the present $p$ is non-bipartite a.a.s.---statement \ref{653.0654312.063.0654312.0} 
is equivalent to saying that for even $n$, every even cycle is a symmetric difference 
of Hamilton-circuits: $G\sim G_{n,p}$ with $p(n)\geq n^{-1/2+\varepsilon}$ 
a.a.s. contains triangles. Since $n$ is even, if there were an even cycle 
$z\in\upZ_1(G;\F_2)$ not in the $\F_2$-span of the Hamilton-circuits of $G$, 
then this cycle together with some triangle would represent two linearly 
independent elements modulo $\langle\mathcal{H}(G)\rangle_{\F_2}$, in contradiction 
to $\dim_{\F_2}\left(\upZ_1(G;\F_2)/\langle\mathcal{H}(G)\rangle_{\F_2}\right)=1$, 
the latter being guaranteed to hold a.a.s. by \ref{653.0654312.063.0654312.0}. 

All graphs considered, 
the property of graphs $G$ with $\upZ_1(G;\F_2) = \langle\mathcal{H}(G)\rangle_{\F_2}$ 
(no Hamilton-connectedness required) is not monotone, so a priori one should 
be cautious of speaking of a threshold for that property in $G_{n,p}$. 
It seems likely, though, that for binomial random graphs $G_{n,p}$ the property 
can only arise `within a monotone property' and that there \emph{is} a threshold. 
Should it be possible to improve Theorem~\ref{dsafrwe5ytrdsf43wrte65egrwterfdx} in 
Section~\ref{dfrg55frgtr655getw435wter6ert456ertr54} to the threshold 
for \emph{Hamilton-connectedness} of $G_{n,p}$, then one would know that 
for any $p$ which a.a.s. ensures 
$\upZ_1(G;\F_2) = \langle\mathcal{H}(G)\rangle_{\F_2}$, 
the property necessarily comes together with Hamilton-connectedness, i.e. then 
one would know that, if a.a.s. $G_{n,p}$ $\in$ 
$\{G\colon$ $\upZ_1(G;\F_2) = \langle\mathcal{H}(G)\rangle_{\F_2}$ $\}$, 
then a.a.s. $G$ $\in$ 
$\{G\colon$ $\upZ_1(G;\F_2) = \langle\mathcal{H}(G)\rangle_{\F_2}$ $\}$ $\cap$ 
$\{G\colon$ $G$ Hamilton-connected $\}$, a monotone graph property. 
If this is the case, then by \cite[Theorem~1.1]{MR1371123} we would know that 
$\upZ_1(G;\F_2) = \langle\mathcal{H}(G)\rangle_{\F_2}$ has a sharp threshold. 
Altghough it is clear that a.a.s. 
$\upZ_1(G;\F_2) = \langle\mathcal{H}(G)\rangle_{\F_2}$ implies that a.a.s. 
any two adjacent vertices are connected by a Hamilton-path, it is not clear that 
it also implies that a.a.s. any two non-adjacent vertices are. 

\section{An upper bound for the smallest sufficient $p$}
\label{fdsgert6545345436544534w}

In this section we prove an upper bound on the smallest $p$ sufficient 
for $\upZ_1(G;\F_2) = \langle\mathcal{H}(G)\rangle_{\F_2}$ to a.a.s. hold in $G_{n,p}$: 

\begin{theorem}
\label{540453.1053.0531.05432.06531.05120}
If $\varepsilon>0$, $p\in[0,1]^{\N}$ with $p(n)\geq n^{-\frac12+\varepsilon}$, 
$\mathcal{H}(G)$ denotes the set of all Hamilton circuits of a graph $G$ 
and $\upZ_1(G;\F_2)$ its cycle space, then for $n\to\infty$ 
a random graph $G \sim G_{n,p}$ asymptotically almost surely 
has the following properties: 
\begin{enumerate}[label={\rm(\arabic{*})},leftmargin=2em]
\item\label{2.04531.078645312084563120}
the $\F_2$-span of all circuits having length $n$ or $n-1$ in $G$ 
is $\upZ_1(G;\F_2)$\quad , 
\item\label{653.0654312.063.0654312.0}
if $n$ is even, then 
$\dim_{\F_2}\left(\upZ_1(G;\F_2)/\langle\mathcal{H}(G)\rangle_{\F_2}\right)=1$ \quad , 
\item\label{45613.0465130.4531.04563120} 
if $n$ is odd, the $\F_2$-span of $\mathcal{H}(G)$ is $\upZ_1(G;\F_2)$ \quad . 
\end{enumerate}
\end{theorem}
\begin{proof}[Proof of Theorem~\ref{540453.1053.0531.05432.06531.05120}]
We use some notation from \cite{arXiv:1112.5101v1}: we denote by $\upC_n^2$ 
the square of an $n$-vertex circuit. 
The symbol $\mathcal{M}_{\{\lvert\cdot\rvert\},0}$ denotes the set of 
all finite graphs which are simultaneously Hamilton-connected and 
have each cycle equal to a symmetric difference of Hamilton circuits, 
while $\mathcal{M}_{\{\lvert\cdot\rvert\},1}$ denotes the set of 
all finite graphs $G$ which are simultaneously Hamilton-connected and have 
$\dim_{\F_2}\left(\upZ_1(G;\F_2)/\langle\mathcal{H}(G)\rangle_{\F_2}\right)=1$. 
The symbol $\mathcal{M}_{\{\lvert\cdot\rvert-1,\lvert\cdot\rvert\},0}$ denotes 
the set of all graphs $G$ which simultaneously have the property that any 
two vertices are connected by a path of length at least $\lvert G\rvert-2$ and 
that the $\F_2$-span of the set of all circuits of length at least $\lvert G\rvert-1$ 
is equal to $\upZ_1(G;\F_2)$. 
By a recent theorem of K{\"u}hn and Osthus 
\cite[Theorem~1.2 specialized to $k=2$]{MR3022146}, 
we know that, asymptotically almost surely, $G_{n,p}$ contains $\upC_n^2$ 
as a subgraph. By (a3), (a4) and (a5) in 
\cite[Lemma~17]{arXiv:1112.5101v1}, 
$\upC_n^2\in\mathcal{M}_{\{\lvert\cdot\rvert\},0}$ if $n$ is odd, 
while both $\upC_n^2\in\mathcal{M}_{\{\lvert\cdot\rvert-1,\lvert\cdot\rvert\},0}$ 
and $\upC_n^2\in\mathcal{M}_{\{\lvert\cdot\rvert\},1}$ if $n$ is even. 
By \cite[Lemma~18.(1)]{arXiv:1112.5101v1}, 
each of the graph properties $\mathcal{M}_{\{\lvert\cdot\rvert\},0}$, 
$\mathcal{M}_{\{\lvert\cdot\rvert-1,\lvert\cdot\rvert\},0}$ and 
$\mathcal{M}_{\{\lvert\cdot\rvert\},1}$ is monotone. 
Therefore, a.a.s. $G_{n,p}$ itself is in $\mathcal{M}_{\{\lvert\cdot\rvert\},0}$ for 
odd $n$ (proving \ref{45613.0465130.4531.04563120}), and in 
both $\mathcal{M}_{\{\lvert\cdot\rvert\},1}$ (proving \ref{653.0654312.063.0654312.0}) 
and $\mathcal{M}_{\{\lvert\cdot\rvert-1,\lvert\cdot\rvert\},0}$ for even $n$. 
The latter completes the proof of \ref{2.04531.078645312084563120}, for although 
\ref{2.04531.078645312084563120} is formulated without a parity condition, in the 
case of odd $n$, \ref{2.04531.078645312084563120} is implied 
by \ref{45613.0465130.4531.04563120} anyway. 
\end{proof}

\section{A lower bound for the smallest sufficient $p$}
\label{dfrg55frgtr655getw435wter6ert456ertr54}

In this section we will derive a lower bound (larger than the hamiltonicity threshold) 
that any $p$ which a.a.s. ensures 
$\upZ_1(G_{n,p};\F_2) = \langle\mathcal{H}(G_{n,p})\rangle_{\F_2}$ must satisfy. 
We prepare with a few lemmas. 

\begin{lemma}\label{dsfgty36yetdfs4r34we4q2334erw54tredr34tre}
If $G=(V,E)$ is a graph such that 
\begin{enumerate}[label={\rm(\arabic{*})},leftmargin=2em]
\item\label{fdg5rsredfr43fredsfrewdsfre} 
$G$ is neither a forest nor a circuit\quad , 
\item\label{vsfgert5y4wre45345ferw54wrq43er} 
every cycle in $G$ is a symmetric difference of Hamilton circuits\quad , 
\item\label{bgfdsrew54ye3rw3424g5534ftre4re} 
$G$ contains a vertex of degree $2$\quad , 
\end{enumerate}
then for every vertex $v$ as in \ref{bgfdsrew54ye3rw3424g5534ftre4re}, 
the graph $G-v$ obtained by deleting $v$ is bipartite. 
\end{lemma}
\begin{proof}
Let $G$ be any such graph. Being a non-forest, the cycle space of $G$ is non-trivial. 
Since otherwise \ref{vsfgert5y4wre45345ferw54wrq43er} fails (for trivial reasons), 
we may assume that $G$ contains a Hamilton circuit, 
in particular, $G$ is $2$-connected. We now choose any $v\in V$ with $\deg(v)=2$. 
If all vertices of $G$ had degree $2$, then $G$ would be one single circuit, 
contradicting hypothesis \ref{csdfgry45tersdferw543ewr54edfs}. 
We may thus assume that there exists $w'\in V$ with $\deg(w')\geq 3$. 
By connectedness, there exists a $v$-$w'$-path $P$, and by finiteness there exists 
a vertex $w$ on this path such that $\deg(w)\geq 3$ and all vertices 
between $v$ and $w$ have degree $2$ in $G$. 
(Possibly, $w$ is a neighbour of $v$ and there are no vertices between $v$ and $w$ 
at all.) Let $v^-$ and $v^+$ denote the two neighbours of $v$, with $v^+$ the one 
in the direction of $w$ along $P$ (possibly, $v^+=w$). Then all vertices from $v$ 
up to and including the predecessor $w^-$ of $w$ on $P$ have degree $2$, hence 
\begin{equation}\label{adfgrw5te4sfaewafr4get65edre453erd}
\text{every circuit in $G$ either contains all or 
none of the edges $v^-v$, $vv^+$, $\dotsc$, $w^-w$}\quad . 
\end{equation}
Now consider some circuit $C$ in $G$ which does not contain 
any of the edges $v^-v$, $vv^+$, $\dotsc$, $w^-w$. Such circuits exist, 
since for example any two neighbours $w',w''$ of the $\geq 2$ neighbours of $w$ 
other than $w^-$ are connected by a path $\tilde{P}$ 
which neither contains $w$ (by $2$-connectedness of $G$ and Menger's theorem) 
nor any of the vertices $v$, $v^+$, $\dotsc$, $w^-$ (since these all have degree $2$), 
so the circuit $ww'\tilde{P}w''w$ is an example. 
By hypothesis \ref{vsfgert5y4wre45345ferw54wrq43er}, there exist Hamilton circuits 
$H_1,\dotsc,H_t$ of $G$ such that $C$ equals their symmetric difference. Each $H_i$ 
contains $v$, hence $vv^+$, hence in view of \eqref{adfgrw5te4sfaewafr4get65edre453erd} 
must contain all edges $v^-v$, $vv^+$, $\dotsc$, $w^-w$. 
Since $C$ itself does not contain any of these edges, $t$ is even. 
We have shown that every circuit $C$ in $G$ which does not contain 
any of the edges $v^-v$, $vv^+$, $\dotsc$, $w^-w$ is the symmetric difference 
of an even number of Hamilton circuits; since every such circuit has an 
even number of edges (no matter what the parity of $\lvert G\rvert$ is), 
and since every circuit in $G-v$ is a circuit not containing 
any of the edges $v^-v$, $vv^+$, $\dotsc$, $w^-w$, 
this proves that $G-v$ is bipartite. 
\end{proof}

\begin{definition}[{$\upK^{\widehat{s},s-1}$}]\label{dsegrw5t4ertsfdwe4t3wer54ewd}
For every $s\geq 3$ we define $\upK^{\widehat{s},s-1}$ to be the graph 
obtained from the complete bipartite graph with classes 
$\{1,3,5,\dotsc,2s-1\}$ and $\{2,4,\dotsc,2s-2\}$ by adding the vertex $0$ 
and the two edges $\{0,1\}$ and $\{0,2s-1\}$. 
\end{definition}

\begin{proposition}\label{dfgrw5ety65erw423454egrt4rtds}
If $G=(V,E)$ is a graph such that 
\begin{enumerate}[label={\rm(\arabic{*})},leftmargin=2em]
\item\label{csdfgry45tersdferw543ewr54edfs} 
$G$ is neither a forest nor a circuit\quad , 
\item\label{cxsdfgertyrsdfw43tgert4trder5red} 
every cycle in $G$ is a symmetric difference of Hamilton circuits\quad , 
\end{enumerate}
then it does \emph{not} follow that $G$ has minimum degree $\geq 3$. 
\end{proposition}
\begin{proof}
We prove that the graph $\upK^{\widehat{4},3}$ 
from Definition~\ref{dsegrw5t4ertsfdwe4t3wer54ewd} is an example for this 
non-implication. Evidenly, it is neither a forest nor a circuit and it 
does not have minimum degree $\geq 3$. So all we have to show is that 
\ref{cxsdfgertyrsdfw43tgert4trder5red} holds for $G=\upK^{\widehat{4},3}$. 

The cycle space of $G$ has dimension 
$\lVert\upK^{\widehat{4},3}\rVert - \lvert\upK^{\widehat{4},3}\rvert + 1$ $=$ 
$14$ $-$ $8$ $+$ $1$ $=$ $7$. Since a vector space does not contain proper 
subspaces of its own dimension, to prove \ref{cxsdfgertyrsdfw43tgert4trder5red} 
it suffices to exhibit seven Hamilton circuits linearly independent over $\F_2$: 
the circuits 
\begin{enumerate}[label={\rm(\arabic{*})}]
\item\label{dsfge5y4tdswer45ert6ert564erw} 
$C_1 := 0,1,4,5,2,3,6,7,0$\quad , 
\item\label{sadfg5tdfgy45erw43wgferw432ew}
$C_2 := 0,1,6,3,4,5,2,7,0$\quad ,
\item\label{fdsget4trdgwe43gdftdsxzcsdfert654re}
$C_3 := 0,1,4,3,2,5,6,7,0$\quad ,
\item\label{fdgst54ew354ewe43rwe54e54ew43ew}
$C_4 := 0,1,2,5,4,3,6,7,0$\quad , 
\item\label{fdgsr54w435g4tytewt4fgd4te4trds43}
$C_5 := 0,1,6,5,2,3,4,7,0$\quad ,
\item\label{fdgsr5ty4rtdf4ger4tewr34fr4345er54t}
$C_6 := 0,1,2,3,6,5,4,7,0$\quad , 
\item\label{fdwer54tefwr3erw5wer43w}
$C_7 := 0,1,2,5,6,3,4,7,0$\quad , 
\end{enumerate}
are Hamilton circuits of $\upK^{\widehat{4},3}$, the matrix of $C_1,\dotsc,C_7$ 
w.r.t. the standard basis of the edge-space of $\upK^{\widehat{4},3}$ is 
\begin{equation}\label{fdsgetytdsaer34t5ty46t4e5w43ewr3re4wew}
\begin{smallmatrix}
& & C_1 & C_2 & C_3 & C_4 & C_5 & C_6 & C_7 \\
& & & & & & & & \\
0,1  &  & 1 & 1 & 1 & 1 & 1 & 1 & 1 \\
0,7  &  & 1 & 1 & 1 & 1 & 1 & 1 & 1 \\
1,2  &  & 0 & 0 & 0 & 1 & 0 & 1 & 1 \\
1,4  &  & 1 & 0 & 1 & 0 & 0 & 0 & 0 \\
1,6  &  & 0 & 1 & 0 & 0 & 1 & 0 & 0 \\
2,3  &  & 1 & 0 & 1 & 0 & 1 & 1 & 0 \\
2,5  &  & 1 & 1 & 1 & 1 & 1 & 0 & 1 \\
2,7  &  & 0 & 1 & 0 & 0 & 0 & 0 & 0 \\
3,4  &  & 0 & 1 & 1 & 1 & 1 & 0 & 1 \\
3,6  &  & 1 & 1 & 0 & 1 & 0 & 1 & 1 \\
4,5  &  & 1 & 1 & 0 & 1 & 0 & 1 & 0 \\
4,7  &  & 0 & 0 & 0 & 0 & 1 & 1 & 1 \\
5,6  &  & 0 & 0 & 1 & 0 & 1 & 1 & 1 \\
6,7  &  & 1 & 0 & 1 & 1 & 0 & 0 & 0 \\
\end{smallmatrix}
\end{equation}
and this matrix indeed has $\F_2$-rank $7$. This completes the proof 
of Proposition~\ref{dfgrw5ety65erw423454egrt4rtds}. 
\end{proof}

\begin{figure} 
\begin{center}
\includegraphics{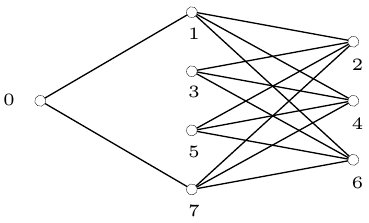}
\caption{The graph $\upK^{\widehat{4},3}$, a small example 
that being neither a forest nor a circuit and having every cycle 
a symmetric difference of Hamilton circuits does not \emph{always} 
imply minimum degree $\geq 3$. This implication holds \emph{a.a.s.} 
in $G_{n,p}$, though, see Lemma~\ref{fdsgrwtytrhdfs4w3w54ert65w43re}.}
\label{fdgw45tdwert54ter5y4terw5rwerre}
\end{center}
\end{figure}

For $s=3$, the graph $\upK^{\widehat{3},2}$ is not an example for the 
non-implication in Proposition~\ref{dfgrw5ety65erw423454egrt4rtds}; 
but for every $s\geq 4$ it seems to be: while we do not need this here, 
let us note in passing that in view of the example $\upK^{\widehat{4},3}$ 
used for proving Proposition~\ref{dfgrw5ety65erw423454egrt4rtds}, 
and also in view of some calculations for the case $s=5$, 
it seems that for \emph{every} $s\geq 4$ the graph $\upK^{\widehat{s},s-1}$ 
has every cycle a symmetric difference of Hamilton circuits, 
despite its degree-$2$-vertex. I.e., it seems that the $\upK^{\widehat{s},s-1}$ 
for $s\geq 4$ are an infinite set of examples for the non-implication 
in Proposition~\ref{dfgrw5ety65erw423454egrt4rtds}: 

\begin{conjecture}\label{dsfg45teft34gth4564gwe54gert4trdfx}
$\upZ_1(\upK^{\widehat{s},s-1};\F_2) 
= \langle\mathcal{H}(\upK^{\widehat{s},s-1})\rangle_{\F_2}$ for every $s\geq 4$. 
\end{conjecture}

We will now derive a probabilistic analogue of 
Proposition~\ref{dfgrw5ety65erw423454egrt4rtds}. We will use: 

\begin{lemma}[{\cite[Exercise~3.2]{MR1864966}}]
\label{dfsgr5etgrsfdwert54egrt4trd}
For any $k\geq 0$ and any $\omega\in\R_{>0}^{\N}$ 
with $\omega_n\xrightarrow[]{n\to\infty}\infty$, 
\begin{enumerate}[label={\rm(\arabic{*})}]
\item\label{fdxcsfgert6hr5trf43wger54ter354ed}
if $p(n)\geq\tfrac{\log n + k\log\log n + \omega_n}{n}$, then 
$\Prob_{G_{n,p}}[\delta\geq k+1]\xrightarrow[]{n\to\infty} 1$ \quad , 
\item\label{cxvsfg5re4erwfdgst54etrdfs5ew}
if $p(n) = \tfrac{\log n + k \log\log n + c + h(n)}{n}$ 
with some constant $c\in\R$ and some function $h\in o(1)$, \\
then 
$\Prob_{G_{n,p}}[\delta = k]\xrightarrow[]{n\to\infty} 1-\exp(-\exp(-c/k!))$ \\
and 
$\Prob_{G_{n,p}}[\delta = k+1]\xrightarrow[]{n\to\infty} \exp(-\exp(-c/k!))$\quad , 
\item\label{fdgserwe45wre543er54ewrsfsdgwer5efds}
if $p(n)\leq\tfrac{\log n + k\log\log n - \omega_n}{n}$, then 
$\Prob_{G_{n,p}}[\delta\leq k]\xrightarrow[]{n\to\infty} 1$ \quad .
\end{enumerate}
\end{lemma}

If $G$ is a forest, the property $\upZ_1(G;\F_2) = \langle\mathcal{H}(G)\rangle_{\F_2}$ 
vacuously holds, and then no conclusions can be drawn from it. This is the reason 
for \ref{dfsgrtrtdfg5egrt45ytrdre5e} of the following lemma. 
As to the second hypothesis in \ref{dfsgrtrtdfg5egrt45ytrdre5e}, of course, 
for any $p$ a.a.s. $G_{n,p}$ is not a circuit, so we could just leave out 
`nor a circuit'; but we leave it in for better analogy with the deterministic 
Proposition~\ref{dfgrw5ety65erw423454egrt4rtds}:

\begin{lemma}[{the non-implication from Proposition~\ref{dfgrw5ety65erw423454egrt4rtds} 
holds for $G_{n,p}$}]\label{fdsgrwtytrhdfs4w3w54ert65w43re}
If $p\in[0,1]^{\N}$ is such that with $G\sim G_{n,p}$ a.a.s. for odd $n$ 
\begin{enumerate}[label={\rm(\arabic{*})}]
\item\label{dfsgrtrtdfg5egrt45ytrdre5e} 
$G$ is neither a forest nor a circuit\quad , 
\item\label{sdfgt5rtdsge4swre5eswre} 
every cycle in $G$ is a symmetric difference of Hamilton circuits \quad , 
\end{enumerate}
then a.a.s. $G$ has minimum degree $\geq$ $3$. 
\end{lemma}
\begin{proof}
Suppose that we do \emph{not} have 
\begin{equation}\label{fdsgretyrtedsfgwer545r5wer}
\Prob_{G_{n,p}}[\delta(G)\geq 3]\xrightarrow[]{\mathrm{odd}\ n\to\infty} 1 \quad . 
\end{equation} 
Then by the Bolzano--Weierstrass-theorem there exists $0<\xi<1$ and a 
subsequence $(n_i)_{i\in\N}$ of odd numbers such that 
\begin{equation}\label{dcsfger45gerw54erth64534243rw54rt}
\Prob_{G_{n_i,p}}[\{ G\subseteq\upK^{n_i}\colon \delta(G)\geq 3\}]
\xrightarrow[]{i\to\infty} \xi \quad  , 
\end{equation}
equivalently 
\begin{equation}\label{asdf4rr2w54ser454w3}
\Prob_{G_{n_i,p}}[\{ G\subseteq\upK^{n_i}\colon \delta(G)\leq 2 \}]
\xrightarrow[]{i\to\infty} 1-\xi \in (0,1) \quad  . 
\end{equation}
Since with our $p$, a.a.s. $G$ contains some circuit and has 
property \ref{sdfgt5rtdsge4swre5eswre}, in particular we a.a.s. have at least 
one Hamilton circuit in $G$, hence in particular we a.a.s. have $\delta(G)\geq 2$, 
i.e., so 
\begin{equation}\label{fdsgretyhtytrw54etrhe56erw5re}
\Prob_{G_{n_i,p}}[\{ G\subseteq\upK^{n_i}\colon \delta(G)\geq 2 \}]
\xrightarrow[]{i\to\infty} 1 \quad  . 
\end{equation}
Twice using the fact that intersecting with an 
a.a.s. property does not change an asymptotic probability, 
\eqref{asdf4rr2w54ser454w3} and \eqref{fdsgretyhtytrw54etrhe56erw5re} 
together imply 
\begin{equation}\label{dfsgr5egersadfsgrht65etgwertw54rtds}
\Prob_{G_{n_i,p}}[\{ G\subseteq\upK^{n_i}\colon \delta(G)= 2 \}]
\xrightarrow[]{i\to\infty} 1-\xi \in (0,1) \quad  , 
\end{equation}
and now \eqref{dfsgr5egersadfsgrht65etgwertw54rtds} 
together with the a.a.s. property 
$\upZ_1(G;\F_2) = \langle\mathcal{H}(G)\rangle_{\F_2}$ 
from hypothesis \ref{sdfgt5rtdsge4swre5eswre} implies 
\begin{equation}\label{dfsgrw5e34yw34345wre}
\Prob_{G_{n_i,p}}[\{ G\subseteq\upK^{n_i}\colon \delta(G)= 2
\quad\mathrm{and}\quad 
\upZ_1(G;\F_2) = \langle\mathcal{H}(G)\rangle_{\F_2}\}]
\xrightarrow[]{i\to\infty} 1-\xi \in (0,1) \quad . 
\end{equation}
All in all we now know that with our $p$, 
\begin{equation}\label{vcsbfgret45serdfasgrt54tsg5t4rtdfsxzzsderw3}
\Prob_{G_{n_i,p}}\left[\left\{ G\subseteq\upK^{n_i}\colon 
\parbox{0.45\linewidth}{$G$ is neither a forest nor a circuit\\
and every cycle in $G$ is a symmetric difference of Hamilton circuits\\ 
and $G$ contains a vertex of degree $2$}\right\}\right]
\xrightarrow[]{i\to\infty} 1-\xi \in (0,1) \quad . 
\end{equation}

From Lemma~\ref{dsfgty36yetdfs4r34we4q2334erw54tredr34tre} 
we know the deterministic implication that, for any $n\in\N$, 
\begin{align}\label{dfsgretgrsersdae4wsfgt6yet54rt65t}
& \left\{ G\subseteq\upK^n\colon 
\parbox{0.45\linewidth}{$G$ is neither a forest nor a circuit\\
and every cycle in $G$ is a symmetric difference of Hamilton circuits\\ 
and $G$ contains a vertex of degree $2$}\right\} \notag \\
\subseteq & \left\{ G\subseteq\upK^n\colon 
\parbox{0.4\linewidth}{for every vertex $v$ of $G$ with $\deg(v)=2$ the graph $G-v$ 
is bipartite}\right\}\quad .
\end{align}
We abbreviate 
\begin{enumerate}[label={\rm(\arabic{*})}]
\item\label{dsafsrt54trsdfgwehty65445g54trs54} 
$\mathcal{B}_{G-v,n}$ $:=$ $\{$ $G\subseteq\upK^n\colon$ 
for every $v\in\upV(G)$ with $\deg(v)=2$ the graph $G-v$ is bipartite $\}$,
\item\label{sdfgtrsdfadreswfredsfed} 
$\mathcal{E}_{\delta=2,n}$ $:=$ $\{$ $G\subseteq\upK^n\colon$ 
there exists in $G$ a vertex $v$ with $\deg(v)=2$ $\}$.
\end{enumerate}

Applying $\Prob_{G_{n_i,p}}$ to \eqref{dfsgretgrsersdae4wsfgt6yet54rt65t}, 
taking $\limsup$ on both sides and using 
\eqref{vcsbfgret45serdfasgrt54tsg5t4rtdfsxzzsderw3}, it follows that 
\begin{equation}\label{vcsvdfg56egetry7645h453er54er}
0 < 1-\xi \leq \limsup_{i\to\infty} 
\Prob_{G_{n_i,p}}\left[ \mathcal{B}_{G-v,n_i} \right]\quad .
\end{equation}
We now claim that from what we know about $p$, it follows that 
\begin{equation}\label{dsfg5ert65derswadferwg54sresd}
\limsup_{i\to\infty} \Prob_{G_{n_i,p}}\left[ \mathcal{B}_{G-v,n_i}\right] = 0 \quad , 
\end{equation}
contradicting \eqref{vcsvdfg56egetry7645h453er54er} and completing the proof 
of Lemma~\ref{fdsgrwtytrhdfs4w3w54ert65w43re}. 

To prove \eqref{dsfg5ert65derswadferwg54sresd}, we first note that 
\eqref{dfsgr5egersadfsgrht65etgwertw54rtds} can also be written 
\begin{equation}\label{adfsgr5serdfr4wre54dwrdf}
\Prob_{G_{n_i,p}}[\mathcal{E}_{\delta=2,n_i}]
\xrightarrow[]{i\to\infty} 1-\xi > 0 \quad ,  
\end{equation}
so we may condition on the event $\mathcal{E}_{\delta=2,n_i}$ and write  
\begin{equation}\label{zdsfgrweytrehdfwt34y653453ew}
\Prob_{G_{n_i,p}}\left[\mathcal{B}_{G-v,n}\right] = 
\Prob_{G_{n_i,p}}\left[\mathcal{B}_{G-v,n}\mid \mathcal{E}_{\delta=2,n}\right]
\cdot \Prob_{G_{n_i,p}}\left[\mathcal{E}_{\delta=2,n_i}\right] \quad . 
\end{equation}
We now claim that from what we know about $p$, 
\begin{equation}\label{fgdret5te42r345t645er54w54wre}
\limsup_{i\to\infty}\Prob_{G_{n_i,p}}\left[
\mathcal{B}_{G-v,n_i}\mid\mathcal{E}_{\delta=2,n_i}\right] =  0 \quad . 
\end{equation}
To prove \eqref{fgdret5te42r345t645er54w54wre}, we prove 
$\limsup_{i\to\infty}$ 
$\Prob_{G_{n_i,p}}\left[\mathcal{B}_{G-v,n_i}\cap\mathcal{E}_{\delta=2,n_i}\right]$ 
$\xrightarrow[]{i\to\infty}$ $0$: we first define 
$\mathcal{E}_{\forall v\colon \triangle\in G-v,n}$ $:=$ $\{$ $G\subseteq \upK^{n_i}\colon$ 
for every vertex $v$ of $G$ the graph $G-v$ contains a triangle $\}$ and note that, 
directly from the definitions
\begin{equation}\label{fdgrety5645665ety4r564534wr45r}
\mathcal{B}_{G-v,n}\cap\mathcal{E}_{\delta=2,n}\cap
\mathcal{E}_{\forall v\colon \triangle\in G-v,n} 
= \emptyset \quad . 
\end{equation} 
We now recall that (cf.~\cite[Theorem~3.4]{MR1782847}, for example) 
if $p(n-1)\gg (n-1)^{-1}$, then $G_{n-1,p}$ a.a.s. contains a triangle. 
Because of \eqref{fdsgretyhtytrw54etrhe56erw5re} and 
Lemma~\ref{fdgserwe45wre543er54ewrsfsdgwer5efds} we must have 
(arbitrarily choosing $\omega_n:=\log\log\log n$), 
$p(n) > \tfrac{\log n + \log\log n - \log\log\log n}{n}$ 
for sufficiently large $n$, hence $p(n) \gg (n-1)^{-1}$. 
Since for our $G\sim G_{n,p}$ we have $G-v\sim G_{n-1,p}$ for every $v\in\upV(G)$, 
it follows that with our $p(n)\gg (n-1)^{-1}$ we have the limit 
\begin{equation}\label{fdsgry4trsdfwgertytrd}
\Prob_{G_{n_i,p}}\left[\mathcal{E}_{\forall v\colon \triangle\in G-v,n_i}\right] 
\xrightarrow[]{i\to\infty} 1 \quad. 
\end{equation}
Again using that intersecting with an a.a.s. property does not change 
an asymptotic probability,  
\begin{align}\label{gwerty6rtsrtf43fret56er4er}
\limsup_{i\to\infty}\Prob_{G_{n_i,p}}\left[
\mathcal{B}_{G-v,n_i}\cap\mathcal{E}_{\delta=2,n_i}\right] 
& \By{\eqref{fdsgry4trsdfwgertytrd}}{=} \limsup_{i\to\infty}\Prob_{G_{n_i,p}}\left[
\mathcal{B}_{G-v,n_i}\cap\mathcal{E}_{\delta=2,n_i}\cap
\mathcal{E}_{\forall v\colon \triangle\in G-v,n_i}\right] \notag \\
& \By{\eqref{fdgrety5645665ety4r564534wr45r}}{=} 0 \quad .
\end{align}
From \eqref{gwerty6rtsrtf43fret56er4er}, \eqref{zdsfgrweytrehdfwt34y653453ew} 
and \eqref{adfsgr5serdfr4wre54dwrdf} 
now indeed follows \eqref{dsfg5ert65derswadferwg54sresd}. As already mentioned, 
this completes the proof of Lemma~\ref{fdsgrwtytrhdfs4w3w54ert65w43re}. 
\end{proof}

In \cite{MR3032903}, Glebov and Krivelevich found the asymptotic behaviour 
of the total number of Hamilton circuits in $G_{n,p}$, starting right from 
the hamiltonicity-threshold $p(n)=\tfrac{\log n + \log \log n + \omega(1)}{n}$. 
An obvious necessary condition to keep in mind when seeking to improve 
Theorem~\ref{540453.1053.0531.05432.06531.05120} is that the total number of 
Hamilton circuits a.a.s. be at least as large as 
$\dim_{\F_2}\upZ_1(G;\F_2) = \lVert G\rVert-\lvert G\rvert+1$. 
On the one hand, if $p(n) = \tfrac{\log n + \log \log n + \omega_n}{n}$ 
with $\omega\in\omega(1)$, 
we a.a.s. have $\lVert G\rVert \sim \tfrac12 n \log n$, hence a.a.s. 
$\dim_{\F_2}\upZ_1(G;\F_2) \sim \tfrac12 n \log n$; on the other hand, by 
\cite[Theorem~1]{MR3032903} and Stirling's approximation, 
$\lvert\mathcal{H}(G_{n,p})\rvert$ $\sim$ 
$(2\pi n)^{\frac12}\left((\log n+\log\log n+\omega_n)/\upe\right)^n(1-o(1))^n$ a.a.s., 
outgrowing $\tfrac12 n \log n$. So the above necessary condition is 
satisfied as soon as $G_{n,p}$ becomes hamiltonian at all. 
This prompts the question whether already right from the threshold for 
hamiltonicity, $G_{n,p}$ a.a.s. has each parity-permitted cycle a
symmetric difference of Hamilton circuits. We now answer that particular question 
in the negative: although in $G_{n,p}$ a.a.s. hamiltonicity necessarily comes with 
extras (e.g., \cite{MR1164841} \cite{MR3032903}) into the bargain, 
in particular it comes with far more Hamilton circuits 
than the dimension of $\upZ_1(G_{n,p};\F_2)$, 
the property $\upZ_1(G_{n,p};\F_2) = \langle\mathcal{H}(G_{n,p})\rangle_{\F_2}$ 
is \emph{not} guaranteed right from the start of a.a.s. hamiltonicity (we recall 
that by \cite[Theorem~1]{MR680304} a.a.s. hamiltonicity already begins 
for $p\geq \tfrac{\log n + \log\log n + \omega_n}{n}$): 

\begin{theorem}[{a Hamilton-generated cycle space does not appear 
right from the onset of hamiltonicity of $G_{n,p}$}]
\label{dsafrwe5ytrdsf43wrte65egrwterfdx}
If $p\in[0,1]^{\N}$ is such that with $G\sim G_{n,p}$ a.a.s. for odd $n$ 
\begin{enumerate}[label={\rm(\arabic{*})}]
\item\label{frw34wer43wer5wreds} 
$G$ is neither a forest nor a circuit\quad , 
\item\label{csvdfgr5etrwr43erwg65erw3rew} 
every cycle in $G$ is a 
symmetric difference of Hamilton circuits \quad , 
\end{enumerate}
then for every $c>0$ we have $p(n) > \tfrac{\log n + 2\log\log n + c}{n}$ 
for all sufficiently large odd $n$. 
\end{theorem}
\begin{proof}
Assume otherwise, i.e. there are $c>0$ and infinitely many odd $n$ with
\begin{equation}\label{zdfegrw54ersdfge4e3gr54egrttr}
p(n)\leq \tfrac{\log n + 2\log\log n + c}{n} \quad .
\end{equation} 
Since we could pass to an infinite subsequence of such $n$ and still argue 
as we do below, we may assume that \eqref{zdfegrw54ersdfge4e3gr54egrttr} 
holds for every odd $n$. We now argue that it is not true that 
a.a.s. $\delta(G)\geq 3$, 
contradicting Lemma~\ref{fdsgrwtytrhdfs4w3w54ert65w43re} and completing 
the proof of Theorem~\ref{dsafrwe5ytrdsf43wrte65egrwterfdx}. 

We define $p^+(n) := \tfrac{\log n + 2\log\log n + c}{n}$ for every $n$. 
By assumption $p(n)\leq p^+(n)$ for every $n$. Since the property $\delta(G)\geq 3$ 
is monotone increasing, 
$\Prob_{G_{n,p}}[\delta\geq 3] \leq \Prob_{G(n,p^+)}[\delta\geq 3]$ for every $n$, so
\begin{equation}\label{vbfdghrt6edfsaegrt6egr54e54erd}
\limsup_{n\to\infty}\Prob_{G_{n,p}}[\delta\geq 3] 
\leq \limsup_{n\to\infty}\Prob_{G(n,p^+)}[\delta\geq 3] \quad . 
\end{equation}

By Lemma~\ref{dfsgr5etgrsfdwert54egrt4trd}.\ref{fdgserwe45wre543er54ewrsfsdgwer5efds} 
with $k:=3$ and $\omega_n:=\log\log n - c$ 
it follows from the definition of $p^+$ that 
\begin{equation}\label{dfsgrwy5trdsgfr4gtde56dgsge45er54ed}
\Prob_{G(n,p^+)}[\delta(G)\leq 3]\xrightarrow[]{n\to\infty} 1\quad .  
\end{equation}

Since intersecting with an a.a.s. property does not change the $\limsup$ 
of the probabilities of the property, 
it follows from \eqref{dfsgrwy5trdsgfr4gtde56dgsge45er54ed} that 
\begin{equation}\label{dsfre456ersdfety4h5e645fs4edrfgdt5rtd} 
\limsup_{n\to\infty}\Prob_{G(n,p^+)}[\delta\geq 3] = 
\limsup_{n\to\infty}\Prob_{G(n,p^+)}[\delta\geq 3\ \mathrm{and}\ \delta\leq 3] = 
\limsup_{n\to\infty}\Prob_{G(n,p^+)}[\delta = 3] \quad .
\end{equation}
The definition of $p^+$ together with 
Lemma~\ref{dfsgr5etgrsfdwert54egrt4trd}.\ref{cxvsfg5re4erwfdgst54etrdfs5ew} 
with $k:=2$ and $h:=0$ ensures that we have the limit 
$\Prob_{G(n,p^+)}[\delta = 3]$ $\xrightarrow[]{n\to\infty}$ $\exp(-\exp(-c/2))$ 
which when substituted into \eqref{dsfre456ersdfety4h5e645fs4edrfgdt5rtd} yields 
\begin{equation}\label{dsfgt5y6edgrtew54ertr6ert54d}
\limsup_{n\to\infty}\Prob_{G(n,p^+)}[\delta\geq 3] = \exp(-\exp(-c/2)) \quad . 
\end{equation}
From \eqref{dsfgt5y6edgrtew54ertr6ert54d} 
and \eqref{vbfdghrt6edfsaegrt6egr54e54erd} we finally get a conclusion for $p$ proper: 
\begin{equation}\label{zfdsgerty4g3tg4tre64erw5erd}
\limsup_{n\to\infty}\Prob_{G_{n,p}}[\delta\geq 3] \leq \exp(-\exp(-c/2)) < 1\quad , 
\end{equation}
for any $c\in\R$. The bound \eqref{zfdsgerty4g3tg4tre64erw5erd} shows that 
with our $p$ it is not the 
case that $\Prob_{G_{n,p}}[\delta\geq 3]\xrightarrow[]{n\to\infty} 1$, 
the contradiction to Lemma~\ref{fdsgrwtytrhdfs4w3w54ert65w43re} already mentioned. 
\end{proof}

\section{Concluding remarks}

It seems very likely that the condition $p(n)\geq n^{-1/2+\varepsilon}$ 
in Theorem~\ref{540453.1053.0531.05432.06531.05120} can be significantly improved: 
apparently (work in progress), by adapting available embedding 
technology to the present purposes one can improve $-1/2$ to $-2/3$ 
in Theorem~\ref{540453.1053.0531.05432.06531.05120}. 
It also seems as if $p(n)\geq n^{-2/3+\varepsilon}$ is the utmost of what can 
be achieved with the technique of 
proving the existence of some spanning subgraph pre-selected from 
the monotone property $\mathcal{M}_{\{\lvert\cdot\rvert\},0}$: the graph $\upC_n^2$ 
is a non-minimal element of the monotone property 
$\mathcal{M}_{\lvert\cdot\lvert,0}$ (resp., for even $n$, 
of $\mathcal{M}_{\lvert\cdot\lvert,1}$), and it is possible to 
descend quite far by deleting edges to construct sparser `rebar', 
and yet $p(n)\geq n^{-2/3+\varepsilon}$ seems to be required even when using 
a minimal element of $\mathcal{M}_{\lvert\cdot\lvert,0}$. 

It nevertheless still seems plausible that one may even 
significantly improve the conjectured bound $n^{-2/3+\varepsilon}$, but this will 
require several new ideas: some customary arguments for showing that 
a set of vectors is a generating system carry with them an `instability' unusual 
for arguments about $G_{n,p_n}$. For example, an argument which were to randomly 
choose a dimension-sized subset of distinct Hamilton-circuits and then go on 
to prove this subset to be a.a.s. linearly independent over $\F_2$ would have 
to be so sensitive to the size of the subset so as to break down 
if only $1+\dim_{\F_2}\ \upZ_1(G;\F_2)$ circuits would have been selected instead. 
Nevertheless, it is not inconceivable that with the right ideas one can improve 
the threshold $n^{-1/2+\varepsilon}$ all the way 
down to what Theorem~\ref{dsafrwe5ytrdsf43wrte65egrwterfdx} allows, 
i.e. to $p(n)\geq \tfrac{\log n + 2 \log \log n + \omega_n}{n}$, 
the threshold for minimum degree $\geq 3$.

\bibliographystyle{amsplain} 
\bibliography{hamiltonspaceofrandomgraph_arXiv_version}

\end{document}